\documentclass[12pt]{amsart}

\calclayout

\usepackage[utf8]{inputenc}
\usepackage{amsmath,amsthm,amssymb}
\usepackage{mathrsfs}

\usepackage{colonequals}
\usepackage{enumitem}
\usepackage{graphicx}
\usepackage[justification=centering]{caption}

\usepackage[colorlinks=true,allcolors=blue,pdftitle={GL with pinning in 3D}]{hyperref}
\usepackage{color}
\usepackage[initials,nobysame,alphabetic,msc-links,backrefs]{amsrefs}

\newtheorem{theorem}{Theorem}[section]

\newtheorem{prop}{Proposition}[section]

\newtheorem{define}{Definition}[section]

\newtheorem{remark}{Remark}[section]

\newtheorem{definition}{Definition}[section]

\newtheorem*{theorem*}{Theorem}
\newtheorem*{lemma*}{Lemma}
\newtheorem*{prop*}{Proposition}
\newtheorem*{define*}{Definition}
\newtheorem*{corollary*}{Corollary}

\numberwithin{equation}{section}

\DeclareMathOperator{\diver}{div}
\DeclareMathOperator{\curl}{curl}

\newcommand{\R}{\ensuremath{\mathbb{R}}}

\newcommand{\C}{\ensuremath{\mathbb{C}}}

\newcommand{\ip}[2]{\langle#1\hspace*{.5mm},#2\rangle}

\newcommand{\ep}{\varepsilon}
\newcommand{\ex}{\mathrm{ex}}
\newcommand{\loc}{\mathrm{loc}}
\newcommand{\lep}{{|\log\ep|}}

\renewcommand{\u}{{\mathsf{u}}}
\newcommand{\A}{{\mathsf{A}}}
\renewcommand{\H}{{\mathsf{H}}}

\newcommand{\pmin}{\ensuremath{\rho_\varepsilon}}

\newcommand{\fen}{\ensuremath{F_{\varepsilon, \pmin}}}

\newcommand{\meisconf}{\ensuremath{\pmin e^{i h_\ex \phi_\ep^0}, h_\ex A^0_\ep}}
\newcommand{\meisdecom}{\ensuremath{\pmin u_\ep e^{i h_\ex \phi_\ep^0},A_\ep+ h_\ex A^0_\ep}}
\newcommand{\err}{\mathfrak{r}}
\newcommand{\minsp}{H^1(\Omega,\C) \times [A_{\ex} + H_{\mathrm{curl}}(\R^3,\R^3) ]}

\definecolor{purple}{rgb}{0.63, 0.36, 0.94}

\newcommand{\de}{\delta}
\newcommand{\RR}{\mathrm{R}}

\newcommand{\dist}{\mathrm{dist}}

\newcommand{\NN}{\mathcal N(\Omega)}
\newcommand{\ga}{\Gamma}
\newcommand{\degone}{{f_0}}
\newcommand{\dd}{\mathrm{d}}
\renewcommand{\(}{\left(}
\renewcommand{\)}{\right)}

\newcommand{\e}{\mathbf e}

\usepackage[foot]{amsaddr}

\title[Upper bound for the First critical field in 3D pinned Ginzburg--Landau]{First Critical Field in the pinned three-dimensional Ginzburg--Landau Model: A matching upper bound}

\author{Carlos Rom\'{a}n}
\address{Facultad de Matem\'aticas e Instituto de Ingenier\'ia Matem\'atica y Computacional, Pontificia Universidad Cat\'olica de Chile, Vicu\~na Mackenna 4860, 7820436 Macul, Santiago, Chile}
\email{carlos.roman@uc.cl}

\date{\today}
\thanks{Funding information: ANID FONDECYT 1231593.}
\thanks{Final version. Accepted for publication in Annales Henri Poincaré.}
\begin{document}
\begin{abstract} 
We continue our study of the first critical field $H_{c_1}$ for extreme type-II superconductors governed by the three-dimensional magnetic Ginzburg--Landau functional with a pinning term $a_\varepsilon$, as introduced in our previous work \cite{diaz-roman-3d}. Building upon the lower bound for $H_{c_1}$ and the characterization of the Meissner solution, we now establish a matching upper bound for $H_{c_1}$, thereby identifying its leading-order behavior. This result confirms the sharpness of the previously derived lower bound and further elucidates the connection between the onset of vorticity and a weighted variant of the \emph{isoflux problem}. Our argument is prompted by the upper bound construction we developed in \cite{RSS2}, based on the Biot--Savart law.
\end{abstract}
\maketitle

\section{Introduction}
\subsection{The problem and brief state of the art}
The magnetic Ginzburg--Landau model \cite{GinLan} has long served as a cornerstone in the mathematical and physical study of superconductivity, providing a phenomenological framework that captures the macroscopic behavior of superconducting materials under applied magnetic fields. Introduced by Ginzburg and Landau in 1950, this model was later shown to arise as a limiting case of the Bardeen--Cooper--Schrieffer (BCS) theory \cite{BCS}, which provides a microscopic description of superconductivity; see \cite{FHSS} for a rigorous derivation. The Ginzburg--Landau model is a fundamental contribution to physics, the development of which led to the awarding of several Nobel Prizes.

In type-II superconductors, the emergence of vortices---localized regions where superconductivity breaks down---is a defining feature. These vortices nucleate when the applied magnetic field exceeds a threshold known as the first critical field, $H_{c_1}$. Understanding the precise onset of vorticity is a central problem. While in two dimensions vortices are described by point singularities, in three dimensions they appear as vortex filaments, whose geometry is more intricate and naturally leads to the use of geometric measure theoretic tools.

This work is the continuation of our study \cite{diaz-roman-3d} of the first critical field for the three-dimensional Ginzburg--Landau functional incorporating a spatially dependent pinning term $a_\ep(x)$, which models material inhomogeneities. This setting reflects realistic superconducting samples where impurities or structural variations modify the local superconducting properties and affect the energy landscape. We refer interested readers to the introduction of \cite{diaz-roman-3d}, where the problem is motivated from both the physical and mathematical points of view and an extensive overview of the relevant literature is provided. Nevertheless, for the sake of self-containment, we review below some fundamental aspects of the model and recall part of the relevant literature.

In type-II superconductors, the motion of vortices induces dissipation and effectively degrades the superconducting state. One way to mitigate this effect is to introduce inhomogeneities in the material, which may act as pinning sites and restrict the motion of vortices. The incorporation of such inhomogeneities into the Ginzburg--Landau model goes back to Likharev \cite{Likha}, who proposed to account for spatial variations of the superconducting properties by allowing the coefficients in the energy functional to depend on the position.

Vortex pinning is a subject of considerable interest in the physics of superconductivity. To mention only a few works relevant to our setting, the review \cite{Kwok2016} discusses how different types of defects and inclusions give rise to complex pinning landscapes, which influence vortex configurations and play an important role in enhancing the critical current, namely the maximal current that can be carried without losing the superconducting behavior. In a more microscopic direction, experimental work on FeSe-based superconductors has provided direct evidence that local defects can act as pinning centers for vortices, giving a microscopic mechanism for vortex immobilization \cite{phys-x-2024}. These works provide concrete physical motivation for considering spatially dependent pinning terms in the Ginzburg--Landau functional, rather than the classical homogeneous model.

For a more detailed discussion of the physics underlying vortex pinning and its implications, see, for instance, \cites{Likha, Tin, chapman-richardson-pinning} and the references therein.  

\medskip
After nondimensionalization of the physical constants, the Ginzburg--Landau functional with pinning, which models an inhomogeneous superconducting sample in an applied magnetic field at fixed temperature below the critical one, is given by
\begin{equation}\label{GLenergy}
	GL_\ep(\u,\A)=\frac12\int_\Omega |\nabla_{\A} \u|^2+\frac{1}{2\ep^2}(a_\ep(x)-|\u|^2)^2+\frac12\int_{\R^3}|\H-H_{\ex}|^2.
\end{equation}
Here, $\Omega$ is a bounded domain of $\R^3$, which we assume to be simply connected with $C^2$ boundary. The function $\u:\Omega\rightarrow \mathbb{C}$ is the \emph{order parameter}; its squared modulus (the relative density of Cooper pairs of superconducting electrons in the BCS theory) indicates the local state of the superconductor. The vector field $\A:\R^3\rightarrow \R^3$ is the electromagnetic vector potential of the induced magnetic field $\H=\curl \A$, and $\nabla_\A$ denotes the covariant gradient $\nabla - i\A$. The external (or applied) magnetic field is given by $H_{\ex}:\R^3\rightarrow \R^3$, which is assumed of the form $H_\ex = h_\ex H_{0,\ex}$, where $H_{0,\ex}$ is a fixed vector field and $h_\ex$ is a tunable real parameter representing the intensity of the external field. The parameter $\ep > 0$ is the inverse of the \emph{Ginzburg--Landau parameter}, usually denoted by $\kappa$, a non-dimensional constant depending only on the material. We are interested in the regime of small $\ep$, corresponding to extreme type-II superconductors. Finally, $a_\ep$ is a function that accounts for inhomogeneities in the material; we assume $a_\ep \in L^\infty(\Omega)$ and that it takes values in $[b,1]$, where $b \in (0,1)$ is a constant independent of $\ep$. Regions where $a_\ep = 1$ correspond to sites without inhomogeneities.

\smallskip
The Ginzburg--Landau model is a $\mathbb{U}(1)$-gauge theory, in which all the physically relevant quantities are invariant under the gauge transformations $\u \to \u e^{i\Phi}$, $\A \to \A + \nabla \Phi$, where $\Phi$ is any sufficiently regular real-valued function. Two important gauge-invariant quantities are the supercurrent and the vorticity, respectively defined, for any sufficiently regular configuration $(\u,\A)$, by
\begin{equation}\label{defvorticity}
	j(\u,\A) = (i\u, \nabla_\A \u), \quad \mu(\u,\A) = \curl j(\u,\A) + \curl \A,
\end{equation}
where $(\cdot,\cdot)$ denotes the scalar product in $\mathbb{C}$ identified with $\mathbb{R}^2$.

\smallskip
As discussed above, a key physical feature of type-II superconductors is the occurrence of vortices in the presence of an applied magnetic field. These vortices are analogous to those in fluid mechanics, but quantized, and correspond to regions where $|\u|$ vanishes. Since $\u$ is complex-valued, they carry a nonzero integer topological degree. In the limit $\ep\to 0$, vortices become co-dimension $2$ topological singularities: point singularities in two dimensions and vortex filaments in three dimensions. The vorticity $\mu(\u,\A)$ concentrates on these singular sets, while a large part of the energy localizes near them. This concentration phenomenon has played a central role in the rigorous mathematical analysis of the model.

The onset of vortices is governed by the main critical values of the intensity of the applied field. Below the first critical field $H_{c_1}=O(|\log\ep|)$, as first established by Abrikosov \cite{Abr}, minimizers are vortexless and the material is in the Meissner phase, where the applied field is expelled from the sample. At $H_{c_1}$, the first vortice(s) appear and the applied field penetrates the superconductor through the vortice(s).
Between $H_{c_1}$ and $H_{c_2}$, the superconducting and normal phases coexist in the sample. As the strength of the applied field increases, so does the number of vortices. The vortices repel each other, while the external magnetic field confines them inside the sample. Superconductivity is eventually lost in the bulk at the second critical field $H_{c_2}=O(\ep^{-2})$, while surface superconductivity may persist up to the third critical field $H_{c_3}=O(\ep^{-2})$. We refer to the classical references \cites{SanSerBook,FouHel} for a detailed mathematical treatment, in the homogeneous two-dimensional setting, of the tools used to analyze the behavior of minimizers in the different regimes determined by these critical fields; see also the recent work \cite{CorKach}.

In this work we focus on the first transition. The first critical field $H_{c_1}$ is rigorously defined as the threshold at which global minimizers change from being vortexless to carrying nontrivial vorticity. We next briefly review some relevant developments in the literature on the Ginzburg--Landau model with pinning, including results concerning this transition.

The two-dimensional Ginzburg--Landau model with pinning has been extensively studied. Early works such as \cite{lin-du}, motivated by models for thin superconducting films with variable thickness \cites{chapman-du-gunzburguer-variable-thickness,chapman-richardson-pinning}, considered vortex configurations in the absence of magnetic effects and with a smooth pinning function independent of $\ep$, where vortices are induced by a Dirichlet boundary condition rather than by an applied magnetic field. The decoupling method of Lassoued and Mironescu \cite{lassoued-mironescu} later provided a key tool to reduce the non-homogeneous potential term to a weighted Ginzburg--Landau functional with homogeneous potential. Subsequent works incorporated magnetic effects and analyzed more specific structures of the pinning term, including periodic, two-state, and shrinking-inclusion regimes \cites{ass,kachmar,dos-santos-2019}. In \cite{diaz-roman-2d}, the first critical field was characterized under minimal assumptions on $a_\ep$; we refer interested readers to the introduction of that paper for a detailed account of the two-dimensional literature on pinned Ginzburg--Landau models.

It is also worth mentioning that rotating superfluids and rotating Bose--Einstein condensates are described by a closely related Gross--Pitaevskii model. In this setting, the gauge field $A$ is absent, and the applied magnetic field $H_\ex$ is replaced by a rotation vector whose intensity plays an analogous role. In regimes of sufficiently low rotation, these models can be analyzed by methods closely related to those developed for the Ginzburg--Landau functional; see \cites{Ser3,aftalionbook,Rou1,Rou2} and the references therein. In particular, the asymptotics of the first critical velocity, the analogue of the first critical field, were obtained in the two-dimensional case in \cites{IgnatMillot1,IgnatMillot2}.

In contrast with the two-dimensional situation, the three-dimensional pinned problem is much less developed. In the homogeneous case, however, the first critical field is by now well understood after a series of works. A first leading-order characterization was obtained by Baldo, Jerrard, Orlandi, and Soner \cite{BalJerOrlSon2}. Through a $\Gamma$-convergence analysis, they identified a limiting mean-field functional and characterized the transition at the level of weak limits of the vorticity: the rescaled vorticity converges to zero below the critical threshold, whereas it has a nontrivial limit above it. This yields the first critical field up to an error of order $o(|\log\ep|)$. An analogous weak characterization was also obtained for the three-dimensional Gross--Pitaevskii model of rotating Bose--Einstein condensates, where the analysis reduces to a Ginzburg--Landau type functional with a non-homogeneous potential term, whose coefficient is regular and independent of $\ep$.

Starting from the $\ep$-level estimates developed in \cite{roman-vortex-construction}, the first critical field in the homogeneous three-dimensional Ginzburg--Landau model was subsequently determined with an error of order $o_\ep(1)$ in \cites{Rom-CMP,roman-sandier-serfaty,RSS2}. A central role in these works is played by the detailed description of vortex filaments at the $\ep$-level and by the associated \emph{isoflux} problem studied in \cites{Rom-CMP,roman-sandier-serfaty,RSS2}; see also \cite{AlaBroMon} for an analogous problem in the special case where $\Omega$ is a ball and $H_{0,\ex}$ is the unit vector along the $z$-axis.

For the pinned three-dimensional Ginzburg--Landau functional, our work \cite{diaz-roman-3d} was the first to address the first critical field. There, we established a lower bound for $H_{c_1}$ and characterized the Meissner solution, the unique (up to gauge transformation) vortexless minimizer of the energy below this threshold. Our analysis suggested that the onset of vorticity is governed by a weighted variant of the isoflux problem, hinting at a connection between the Ginzburg--Landau framework and geometric optimization, which is influenced by the presence of the pinning function. However, the absence of a matching upper bound left open the question of whether this optimization problem truly captures the leading-order behavior of $H_{c_1}$.

In the present work, we continue this investigation by establishing a matching upper bound for $H_{c_1}$, thereby identifying its leading-order behavior. This result confirms the sharpness of the lower bound derived in \cite{diaz-roman-3d} and completes the asymptotic characterization (to leading order) of the first critical field in the presence of pinning. Our argument is prompted by the upper bound construction we developed in \cite{RSS2}, which is based on the Biot--Savart law. This approach allowed us to capture not just the leading-order contribution of the (homogeneous) Ginzburg--Landau energy in the presence of prescribed bounded vorticity, but in fact the energy up to an $o_\ep(1)$ error. It serves as a foundation for the present analysis in the weighted Ginzburg--Landau setting.

The upper bound in the present work is obtained by constructing a test configuration in which vorticity concentrates along smooth curves that are nearly optimal for the weighted isoflux problem. We show that the energy of this configuration becomes favorable once the applied field exceeds a threshold matching to leading order the lower bound provided in \cite{diaz-roman-3d}. This confirms that the transition from the Meissner phase to the vortex phase is characterized by the weighted isoflux problem. Our analysis also reveals that the pinning term $a_\ep(x)$ plays a subtle role in shaping the vortex landscape, influencing both the location and structure of the filaments.

Combined with the results of \cite{diaz-roman-3d}, our work thus provides a complete asymptotic description of the first critical field in the pinned three-dimensional Ginzburg--Landau model. This description aligns with the one we previously established for the analogous problem in the two-dimensional setting \cite{diaz-roman-2d}, as well as with the homogeneous three-dimensional case \cites{AlaBroMon,BalJerOrlSon2,Rom-CMP}.

\subsection{Main results}
In order to state our first result, we recall that the analysis of \eqref{GLenergy} can be reduced to analyzing the weighted Ginzburg--Landau functional (see \cite{diaz-roman-3d}*{Section 2})
\begin{equation}\label{GLenergy2}
\widetilde{GL}_\ep(u,A)=\frac12\int_\Omega \pmin^2|\nabla_A u|^2+\frac{\pmin^4}{2\ep^2}(1-|u|^2)^2+\frac12\int_{\R^3}|H-H_{\ex}|^2,
\end{equation}
where the weight $\pmin$ is the unique positive real-valued function that satisfies 
\begin{equation}\label{pde rho}
\left\{
\begin{array}{rcll}
-\Delta \pmin &=& \dfrac{\pmin}{\varepsilon^2}(a_\varepsilon-\pmin^2) &\mathrm{in}\ \Omega\\
\dfrac{\partial \pmin}{\partial \nu} &=& 0&\mathrm{on}\ \partial \Omega.
\end{array}
\right.
\end{equation}
We remark that $\u$ and $u$ are related by $\u=\rho_\ep u$, while $\A=A$. One straightforwardly verifies that $b\leq \pmin^2\leq 1$; see Proposition~\ref{prop-rho_ep}. We will assume throughout this article that $a_\ep$ is such that there exist $\alpha\in (0,1)$, $N>0$, and $C_1>0$ (that do not depend on $\ep$) such that
\begin{equation}\label{hypothesispmin}
	\|\pmin\|_{C^{0,\alpha}(\Omega)}\leq C_1|\log\ep|^N.
\end{equation}

Our first result establishes an upper bound for the free-energy functional associated with \eqref{GLenergy2}, that is, the functional considered in the absence of an external field. It provides an explicit configuration for which the vorticity concentrates on a given curve $\ga$ parametrized by arc length. This construction is optimal at leading order and shows that the energetic cost of such a vortex line is, to leading order, bounded above by $\pi |\pmin^2 \ga|\lep$, where
\begin{equation}\label{length}
|\pmin^2 \ga| \colonequals \int_{0}^{|\ga|} \pmin^2(\ga(s))\, \dd s.
\end{equation}

\begin{theorem}\label{thm:upperbound}
Let $\ga$ be a $C^2$ simple open curve in $\Omega$, parametrized by arc length, that intersects $\partial\Omega$ transversally. Assume \eqref{hypothesispmin} holds. Then, for any $\ep>0$ sufficiently small, there exists $(u_\ep,A)\in H^1(\Omega,\C)\times H^1(\R^3,\R^3)$ such that
\begin{multline}\label{construction:upperbound}
\fen(u_\ep,A)\colonequals \frac12\int_\Omega\pmin^2|\nabla_A u_\ep|^2+\frac{\pmin^4}{2\ep^2}(1-|u_\ep|^2)^2+\frac12\int_{\R^3}|\curl A|^2
\\ \leq \pi |\pmin^2 \ga|\lep +\left(\frac{N+1}{\alpha}+1\right)\pi |\ga|\log\lep+C_{\Omega,\ga}+o_\ep(1),
\end{multline}
where $C_{\Omega,\ga}$ is the constant defined in \eqref{comega}.

In addition, it holds that, for any $\beta\in (0,1]$, we have
\begin{equation}\label{construction:vorticityestimate}
\|\mu(u_\ep,A)-2\pi \ga\|_{(C_T^{0,\beta}(\Omega,\R^3))^*}\leq C\ep^{\frac23 \beta }\lep^{1-\frac23 \beta},
\end{equation}
where $C_T^{0,\beta}(\Omega,\R^3)$ denotes the space of $\beta$-H\"older continuous vector fields defined in $\Omega$ whose tangential component vanishes on $\partial \Omega$.
\end{theorem}

We now turn to our main result on the first critical field. For that, we briefly recall two key ingredients from our previous work \cite{diaz-roman-3d}: 
\begin{itemize}[leftmargin=10pt,labelsep=*,itemindent=*,itemsep=10pt,topsep=10pt]
\item \emph{Meissner state:} The Ginzburg--Landau model admits a unique configuration, modulo gauge transformations, that we refer to as the Meissner state, in reference to the Meissner effect in physics. This state is obtained by minimizing the Ginzburg--Landau energy $GL_\ep(\u,\A)$ under the constraint $|\u| = \rho_\ep$. In the gauge where $\diver \A = 0$ in $\R^3$, the Meissner state takes the form
\begin{equation*}
(\meisconf),
\end{equation*}
where $\phi_\ep^0$ and $A_\ep^0$ depend on the domain $\Omega$, the applied field $H_{0,\ex}$, and the weight function $\pmin$.

Although this configuration is not a true critical point of \eqref{GLenergy}, it provides a good approximation of the unique minimizer (still in the same gauge) below the first critical field, as $\ep \to 0$. 

Closely related is a special vector field $B_\ep^0 \in C_T^{0,1}(\Omega,\R^3)$, which satisfies
\begin{equation*}
	A_\ep^0-\nabla \phi_\ep^0 = \frac{\curl B_\ep^0}{\pmin^2}\quad \mbox{in}\ \Omega,
\end{equation*} 
with $\diver B_\ep^0 = 0$ in $\Omega$ and $B_\ep^0 \times \nu = 0$ on $\partial \Omega$.

The regularity of the vector field $B^0_\ep$, fundamental to our analysis, mainly depends on the regularity of $H_{0,\ex}$ in $\Omega$. In particular, assuming henceforth $H_{0,\ex} \in L^3(\Omega,\R^3)$, we have that 
$$
\|B_\ep^0\|_{C_T^{0,\gamma}(\Omega,\R^3)}\leq C\quad \mbox{for any }\gamma\in(0,1),
$$
where the constant $C>0$ does not depend on $\ep$; see \cite{diaz-roman-3d}*{Proposition 2.2}.

\item \emph{Weighted isoflux problem:} We let $\NN$ be the space of normal $1$-currents supported in $\overline\Omega$, with boundary supported on $\partial\Omega$. We denote 
by $|\cdot |$ the mass of a current. Recall that normal currents are currents with finite mass, whose boundaries have finite mass as well.
We also let $X$ denote the class of currents in $\NN$ that are simple oriented Lipschitz curves. 
An element of $X$  must either be a loop contained in $\overline\Omega$ or have its two endpoints on $\partial \Omega$. 

For any vector field $B\in C_T^{0,1}(\Omega,\R^3)$ and any $\ga\in\NN$ we denote by $\ip{\ga}{B}$ the value of $\ga$ applied to $B$, which corresponds to the circulation of the vector field $B$ on $\ga$ when $\ga$ is a curve.

We also define $\pmin^2\Gamma\in \NN$ by duality, that is, for any 1-form $\omega$ supported in $\Omega$, we let
\begin{equation*}
	\ip{\pmin^2\Gamma}{\omega} \colonequals \ip{\Gamma}{\pmin^2\omega}.
\end{equation*} 

\begin{define}[Weighted isoflux problem for type-II superconductivity]
The weighted isoflux problem is the question of maximizing over $\NN$ the ratio 
\begin{equation*}
	\RR_{\pmin^2}(\ga):=\dfrac{\ip{\ga}{B_\ep^0}}{|\pmin^2\ga|}.
\end{equation*}
\end{define}
Let us remark that the existence of maximizers for $\RR_{\pmin^2}$ is ensured by the weak-$\star$ sequential compactness of $\NN$. Also, let us notice that in the case $\ga$ is a smooth curve parametrized by arc length, \eqref{length} holds.
\end{itemize}
In \cite{diaz-roman-3d} we showed that the occurrence of vortex filaments is possible only if $h_\ex\geq H_{c_1}^\ep$, where
\begin{equation*}
	H_{c_1}^\ep\colonequals \frac{|\log\ep|}{2\RR_\ep}\quad \mbox{and}\quad 
	\RR_\ep\colonequals\sup_{\ga\in X}\RR_{\pmin^2}(\ga)=\sup_{\ga\in X}\frac{\ip{\ga}{B_\ep^0}}{|\pmin^2\ga|},
\end{equation*}
showing in particular that, for some constant $K_0>0$ independent of $\ep$, one has
$$
H_{c_1}^\ep-K_0\log \lep\leq H_{c_1}.
$$ 
Our main result on this paper complements this by providing an upper bound that matches this lower bound at leading order, that is, 
$$
H_{c_1}\leq H_{c_1}^\ep+K^0\log\lep,
$$ 
for some constant $K^0$ independent of $\ep$. 

We need the following assumptions. First, we assume that, for any $\ep$ sufficiently small, there exists a $C^2$ open simple curve $\ga_\ep$ parametrized by arc length which intersects $\partial\Omega$ transversely and such that
\begin{equation}\label{almostopt}
\RR_{\pmin^2}(\ga_\ep)=\RR_\ep+O_\ep(\lep^{-1}).
\end{equation}
We assume in addition that there exists a positive constant $C_0$ (independent of $\ep$) such that
\begin{equation}\label{boundlength}
|\ga_\ep|\leq C_0
\end{equation}  
and
\begin{equation}\label{boundCOmega}
C_{\Omega,\ga_\ep}\leq C_0 \log\lep,
\end{equation}
where $C_{\Omega,\ga_\ep}$ is the constant defined in \eqref{comega}.
\begin{theorem}\label{thm:firstcriticalfield} Assume \eqref{hypothesispmin}, \eqref{almostopt}, \eqref{boundlength}, \eqref{boundCOmega}, and that $\liminf\limits_{\ep\to 0}\RR_\ep>0$. Then there exist $\ep_0>0$ and $K^0>0$ such that, for any $\ep<\ep_0$, any $\eta\in \(0,\frac12\)$, and any 
\begin{equation}\label{hypothesish_ex}
H_{c_1}^\ep+K^0\log\lep\leq h_\ex\leq \ep^{-\eta},
\end{equation}
the global minimizers $(\u,\A)$ of $GL_\ep$ do have vortices. 
\end{theorem}
Let us comment on the role of the hypotheses \eqref{hypothesispmin}, \eqref{almostopt}, \eqref{boundlength}, and \eqref{boundCOmega}. These assumptions have been chosen so as to capture the conditions that are essential for our method and that allow us to control the weighted isoflux problem as $\ep\to0$. More precisely, assumptions \eqref{almostopt}, \eqref{boundlength}, and \eqref{boundCOmega} ensure that the curves associated with the almost optimal weighted isoflux problem do not degenerate as $\ep\to0$. Among them, the uniform length bound \eqref{boundlength} plays a particularly important role: when applying Theorem~\ref{thm:upperbound} to families of curves satisfying \eqref{boundlength}, the $o_\ep(1)$ term in \eqref{construction:upperbound} and the constant $C$ in \eqref{construction:vorticityestimate} are uniform with respect to the curves; see Remark~\ref{rmk:uniformity}.

In the special situation when $a_\ep=a(x)$ is a $C^1$ function that does not depend on $\ep$ and is constant near $\partial\Omega$, when $\Omega$ is a ball and $H_{0,\ex}$ is the unit vector along the $z$-axis, the hypotheses above are satisfied, provided that $a$ is sufficiently close to $1$ in a uniform sense. Indeed, in this case $\pmin^2$ converges uniformly to $a$ as $\ep\to0$, and the weighted isoflux problem is essentially reduced to a fixed weighted problem with a regular weight, independent of $\ep$. When $a\equiv1$, the corresponding non-weighted isoflux problem was analyzed in \cites{roman-sandier-serfaty,RSS2}, where the relevant optimal curve was shown to satisfy suitable non-degeneracy conditions. Therefore, under this closeness assumption, the weighted problem can be viewed as a perturbation of this non-degenerate situation, and the hypotheses on the curves $\ga_\ep$ follow.

Let us emphasize, however, that the weak and strong non-degeneracy conditions proved in \cite{roman-sandier-serfaty}*{Theorem 2} and \cite{RSS2}*{Proposition 3.6}, respectively, are highly nontrivial. These results are precisely what underlies the above perturbation argument: when $a$ is sufficiently close to $1$, the weighted isoflux problem is a small perturbation of a non-degenerate homogeneous problem. Making this stability argument fully explicit would require revisiting the delicate geometric analysis developed in \cites{roman-sandier-serfaty,RSS2}. Since the example above is only meant to illustrate that our assumptions are compatible with natural situations close to the homogeneous case, we do not pursue this verification here.

We nonetheless expect the hypotheses above to hold in a much more general setting, although their verification depends on the particular model of $a_\ep$ under consideration. For instance, if $a_\ep$ exhibits oscillations as $\ep\to0$, the limiting behavior of the weighted isoflux problem is expected to be related to homogenization phenomena. This is an interesting direction for future work, but it lies beyond the scope of the present article.

On the other hand, in \cite{diaz-roman-3d}*{Proposition 1.2}, we provided a sufficient condition for $\liminf\limits_{\ep\to 0} \RR_\ep > 0$ to hold.

\subsection*{Plan of the paper} 
The rest of the paper is organized as follows. In Section~\ref{sec:preliminaries}, we provide some preliminary results on tubular neighborhoods of a curve, the Biot--Savart vector field associated with a curve, an energy splitting formula, and bounds on $\rho_\ep$. In Section~\ref{sec:proofmain1}, we present a proof of Theorem~\ref{thm:upperbound}. Finally, in Section~\ref{sec:proofmain2}, we provide a proof of Theorem~\ref{thm:firstcriticalfield}.

\subsection*{Acknowledgments} This work was partially funded by ANID FONDECYT 1231593. The author wishes to thank the anonymous referees for their careful reading and useful suggestions, which helped improve this paper.

\section{Some preliminaries}\label{sec:preliminaries}

\subsection{Tubular neighborhood of a curve}\label{sec:coordinates}
Let $\ga:[0,|\ga|]\to \R^3$ be a $C^2$ simple curve parametrized by arc length. We define $\e_1(s)$, $\e_2(s)$  such that $(\ga'(s),\e_1(s),\e_2(s))$ is a direct orthonormal basis for any $s$ and $C^1$ smooth with respect to $s$.

We define the regular tubular neighborhood $T_\de(\ga)$ of $\ga$ by
$$
T_\de(\ga)=\{\ga(s)+y \ | \ s\in[0,|\ga|],\ y\perp \ga'(s),\ |y|<\de\},
$$
where its thickness $\de>0$ is assumed to be sufficiently small so that the normal projection $\Pi:T_\de(\ga)\to \ga([0,|\ga|])$ is well defined and $C^1$. In $T_\de(\ga)$, we make use of curvilinear coordinates $(s,v,w)$, that is
\begin{equation}\label{coordinates}
	x=\ga(s)+ x^\perp(s,v,w),\quad x^\perp(s,v,w) = v\e_1(s)+w\e_2(s),
\end{equation}
for any $x=x(s,v,w)\in T_\de(\ga)$. 
The volume element in these coordinates is 
\begin{equation}\label{volume}
	\dd V=|1-x^\perp\cdot(\ga')'|\dd s\dd v\dd w=(1+O(\de))\dd s\dd v\dd w.
\end{equation}

\subsection{Biot--Savart vector field}
We next recall a result proved in \cite{RSS2}*{Section 2}, concerning the Biot--Savart vector field associated to a smooth simple closed curve $\ga$ in $\R^3$, which is defined as 
\begin{equation*}
X_\ga(p) = \frac12\int_I \frac{\ga(t) - p}{|\ga(t) - p|^3}\times \ga'(t)\,dt,
\end{equation*}
where $\ga:I\to\R^3$ parametrizes the curve.

We recall that $X_\ga$ is divergence-free, satisfies 
\begin{equation}\label{eqXga}
	\curl X_\ga = 2\pi\ga\quad \mathrm{in}\ \R^3,
\end{equation}
and belongs to $L^p_\loc(\R^3,\R^3)$ for any $1\le p<2$.

Moreover, if we let $p_\ga$ denote the nearest point to $p$ on $\ga$,
\begin{equation}\label{approxBS}X_\ga(p) - \frac{p_\ga - p}{|p_\ga - p|^2}\times\ga'(p_\ga)
\end{equation}
is in $L^q(T_\delta(\ga),\R^3)$, for any $q\ge 1$.
 
\begin{prop}\label{jA}
Assume $\ga$ is a $C^2$ simple closed curve in $\R^3$ which intersects $\partial\Omega$ transversally.
Then there exists a  unique  divergence-free $j_\ga:\Omega\to \R^3$, belonging to $L^p(\Omega,\R^3)$ for any $p<2$, and a unique divergence-free $A_\ga\in H^1(\R^3,\R^3)$ such that
$$
\left\{ 
\begin{array}{rcll}
\curl(j_\ga+A_\ga)&=& 2\pi\ga&\mathrm{in}\ \Omega\\
\nu\cdot j_\ga &=& 0&\mathrm{on}\ \partial\Omega
\end{array}
\right.
$$
and such that
$$
-\Delta A_\ga  = j_\ga\mathbf 1_\Omega \quad \mathrm{in}\ \R^3
$$
holds in the sense of distributions. 
In particular, $j_\ga$ and $A_\ga$ only depend on $\ga\cap\Omega$. It also holds that $A_\ga\in W^{2,\frac32}_\loc(\R^3,\R^3)$, that $j_\ga-X_\ga \in W^{1,q}(\Omega,\R^3)$ for any $q<4$, and that 
\begin{equation}\label{eqjOmega}
j_\ga-X_\ga+A_\ga=\nabla f_\ga\quad \mathrm{in}\ \Omega
\end{equation}
for some harmonic function $f_\ga\in W^{1,q}(\Omega)$ for any $q<4$.
\end{prop}

An important role in the upper bound construction will be played by the following constant.

\begin{definition} Let $\ga$ be a $C^2$ simple closed curve in $\R^3$ that intersects $\partial\Omega$ transversally. We apply Proposition~\ref{jA} and define
\begin{equation}\label{comega}
C_{\Omega,\ga} = \frac12\int_{\R^3}|\curl A_\ga|^2 + \lim_{\rho\to 0}\left(\frac12\int_{\Omega\setminus \overline{T_\rho(\ga)}} |j_\ga|^2 +\pi|\ga|\log\rho\right),
\end{equation}
where $|\ga|$ denotes the length of $\ga$ in $\Omega$.
\end{definition}

\subsection{Energy splitting}
Throughout this paper, we assume that $H_\ex\in L^2_{\loc} (\R^3, \R^3)$ is such that $\diver H_\ex=0$ in $\R^3$, consistent with the non-existence of magnetic monopoles. Consequently, there exists a vector potential $A_\ex \in H^1_{\loc}(\R^3,\R^3)$ such that 
$$\curl A_\ex= H_\ex\quad \text{and} \quad \diver A_\ex=0 \ \text{in} \ \R^3.$$
The natural space for minimizing $GL_\ep$ is 
$H^1(\Omega, \C) \times [A_\ex+ H_{\curl}]$, where 
$$H_{\curl}\colonequals \{ \A \in H^1_{\loc} (\R^3, \R^3) | \curl \A \in L^2(\R^3,\R^3)\};$$ see \cite{Rom-CMP}. 

We next recall the energy splitting provided in \cite{diaz-roman-3d}*{Proposition 1.1}, which plays a crucial role in the proof of Theorem~\ref{thm:firstcriticalfield}.
\begin{prop}
Given any configuration $(\u,\A)\in \minsp$, letting $(u,A)$ be defined via the relation $(\u,\A) = (\pmin u e^{i h_\ex \phi_\ep}, A+h_\ex A^0_\ep)$, we have
\begin{equation}\label{energy splitting equation}
GL_\ep(\u,\A) = GL_\ep(\pmin e^{i h_\ex \phi_\ep}, h_\ex A^0_\ep) + \fen(u,A) -h_\ex \int_\Omega \mu(u,A) \cdot B^0_\ep + \err,
\end{equation}
where
\begin{equation}\label{err}
\err \colonequals \frac{h_\ex^2}{2}\int_\Omega \frac{|\curl B^0_\ep|^2}{\pmin^2}(|u|^2-1).
\end{equation}
\end{prop}

\subsection{Bounds on \texorpdfstring{$\pmin$}{ρ}}
\begin{prop}\label{prop-rho_ep}
	We have $b \leq \rho_\ep^2 \leq 1$.
\end{prop}

\begin{proof}
	Observe that $(\pmin-1)_+=\max\left\{\pmin-1,0\right\} \in H^1(\Omega)$. Testing \eqref{pde rho} against $(\pmin-1)_+$ and integrating by parts, we obtain
	$$
	0 \leq \ep^2 \int_{\{\pmin>1\}} |\nabla \pmin|^2 
	= \int_{\{\pmin>1\}} (\pmin-1)_+\pmin (a_\varepsilon-\pmin^2).
	$$
	Recalling that $a_\ep$ takes values in $[b,1]$, we see that $ (\pmin-1)_+\pmin (a_\varepsilon-\pmin^2) < 0$ when $\pmin > 1$. This implies that $|\{\pmin > 1\}| = 0$, and hence $\pmin \leq 1$.

	The lower bound follows similarly by using $(\pmin-\sqrt{b})_-=-\min\left\{\rho_\ep-\sqrt{b},0\right\}$ as a test function for \eqref{pde rho}.
\end{proof}

\section{Upper bound for the free energy}\label{sec:proofmain1}
In this section we provide a proof for Theorem~\ref{thm:upperbound}.
\begin{proof}[Proof of Theorem~\ref{thm:upperbound}] 
We first let $r_\ep\colonequals\lep^{-q}$, for some $q>1$ to be fixed later, and note that $\ga$ can be extended to a $C^2$ simple closed curve in $\R^3$, with the extension being supported on the complement of $\overline{\Omega}$. Throughout this proof, although the curve $\ga$ has been extended, its length is consistently measured with respect to its original definition within $\Omega$.

We will construct a complex-valued function $u_\ep$ of the form 
$$
u_\ep(x)\colonequals |u_\ep|(x) e^{i\varphi(x)}.
$$ 
Recalling Section~\ref{sec:coordinates}, we start by defining its modulus as
$$
|u_\ep|(x)=
\left\{
\begin{array}{cl}
	\dfrac1{\degone \left(\frac{r_\ep}\ep\right)}\degone \(\dfrac{\dist(x,\ga)}\ep \) & \mbox{if }x\in T_{r_\ep}(\ga)\\
	1&\mbox{if }x\in \Omega \setminus T_{r_\ep}(\ga),
\end{array}
\right.
$$
where hereafter $\degone$ denotes the modulus of the degree-one radial vortex solution $u_0$ that appears in \cite{SanSerBook}*{Proposition 3.11}. We recall from this book that $f_0$ is increasing, $f(0)=0$, $\lim\limits_{R\to \infty}\degone(R)\to 1$, and 
\begin{equation}\label{asymptoticsf}
\lim_{R\to \infty}\left( \frac12 \int_0^R \(|\degone'|^2+\frac{\degone^2}{r^2}+\frac{(1-\degone^2)^2}2 \)r\dd r- \frac{1}{2\pi}\left(\pi \log R+\gamma\right) \right)=0,
\end{equation}
where $\gamma>0$ is the universal constant first introduced in the seminal work \cite{BetBreHel}. 

We now proceed to define the phase of $u_\ep$. We let $\varphi$ be defined via
\begin{equation}\label{derivativephi}
	\nabla \varphi= X_\ga+\nabla f_\ga\quad \mathrm{in}\ \Omega,
\end{equation}
where $X_\ga$ and $f_\ga$ are defined by \eqref{approxBS} and by applying Proposition~\ref{jA} through \eqref{eqjOmega}. Let us remark that since $\curl X_\ga=2\pi\ga$ in $\R^3$, if $\sigma$ denotes a smooth, simple, and closed curve that does not intersect the curve $\ga$, then, by Stokes' theorem, 
$$
\int_\sigma \left(X_\ga+\nabla f\right)=2\pi m,\quad \mbox{for some }m\in \mathbb{Z},
$$
which implies that $X_\ga+\nabla f_\ga$ is exact modulo $2\pi$, and thus admits a scalar potential $\varphi$ satisfying \eqref{derivativephi}, defined modulo $2\pi$.

Finally, we apply once again Proposition~\ref{jA}, in order to set
$$
A(x)=A_{\ga}(x).
$$

The rest of the proof is divided in several steps.

\begin{enumerate}[label=\textsc{\bf Step \arabic*.},ref=\arabic*,leftmargin=0pt,labelsep=*,itemindent=*,itemsep=10pt,topsep=10pt]
	
\item\label{Step1} \textbf{Estimating $|\nabla_A u_\ep|^2-|\nabla u_\ep|^2$ over $T_{r_\ep}(\ga)\cap \Omega$.}	

A straightforward computation shows that
$$
|\nabla_A u_\ep|^2-|\nabla u_\ep|^2=|u_\ep|^2\left(|A|^2-2\nabla \varphi\cdot A\right).
$$
We denote hereafter 
\begin{equation}\label{TOmega}
	T_{r_\ep}^\Omega(\ga)\colonequals T_{r_\ep}(\ga)\cap \Omega.
\end{equation}
Since $|u_\ep|^2 \leq 1$, we apply H\"older's inequality to obtain
$$
\left| \int_{T_{r_\ep}^\Omega(\ga)} \left(|\nabla_A u_\ep|^2-|\nabla u_\ep|^2\right)\right|\leq \|A\|^2_{L^2(T_{r_\ep}^\Omega(\ga),\R^3)} + 2\|\nabla \varphi\|_{L^{\frac32}(T_{r_\ep}^\Omega(\ga),\R^3)} \|A\|_{L^3(T_{r_\ep}^\Omega(\ga),\R^3)}.
$$
Since $A\in W^{2,\frac32}(\Omega,\R^3)$, from Sobolev embedding it follows that $A\in L^r(\Omega,\R^3)$ for any $r\geq 1$. In particular, from the Cauchy--Schwarz inequality, we obtain
$$
\|A\|_{L^3(T_{r_\ep}^\Omega(\ga),\R^3)}\leq \|A\|_{L^6(T_{r_\ep}^\Omega(\ga),\R^3)}|T_{r_\ep}^\Omega(\ga)|^\frac16\leq Cr_\ep^\frac13=o_\ep(1).
$$
Moreover, $\nabla \varphi \in L^r(\Omega,\R^3)$ for any $r<2$, which follows from Proposition \ref{jA}. Hence, 
\begin{equation}\label{covariantderivativeenergy}
	\left| \int_{T_{r_\ep}^\Omega(\ga)} \left(|\nabla_A u_\ep|^2-|\nabla u_\ep|^2\right)\right|\leq C\|A\|_{L^3(T_{r_\ep}^\Omega(\ga),\R^3)}=o_\ep(1).
\end{equation}

\item\label{Step2} \textbf{Energy estimate over $T_{r_\ep}^\Omega(\ga)$.}

We claim that
\begin{equation}\label{energyestimatesmalltube}
	\frac12\int_{T_{r_\ep}^\Omega(\ga)}\pmin^2|\nabla u_\ep|^2+\frac{\pmin^4}{2\ep^2}(1-|u_\ep|^2)^2\leq \pi|\pmin^2\ga|\log \frac{r_\ep}{\ep}+\gamma|\pmin^2\ga|+o_\ep(1),
\end{equation}
where $\gamma$ is the constant that appears in \eqref{asymptoticsf}.

Let us start by observing that, since $\pmin\leq 1$, we have
$$
\frac12\int_{T_{r_\ep}^\Omega(\ga)}\pmin^2|\nabla u_\ep|^2+\frac{\pmin^4}{2\ep^2}(1-|u_\ep|^2)^2\leq
\frac12\int_{T_{r_\ep}^\Omega(\ga)}\pmin^2\left(|\nabla u_\ep|^2+\frac{1}{2\ep^2}(1-|u_\ep|^2)^2\right).
$$
Next, given a point $p\in T_{r_\ep}(\ga)$, we denote by $p_{\ga}$ its nearest point on $\ga$ and let
$$
D_s\colonequals \left\{p\in T_{r_\ep}(\ga) \ :\ p_{\ga}=\ga(s) \right\}.
$$
Notice that $D_s$ is a disk in $\R^2$ of radius $r_\ep$ centered at $\ga(z)$. 

Recalling \eqref{approxBS}, we find that
$$
h_{\ga}(p)\colonequals X_\ga(p) - Y_\ga(p)\in L^s\left( T_{r_\ep}(\ga),\R^3\right)
$$
for any $s\geq 1$, where
$$
Y_{\ga}(p)\colonequals \frac{p_{\ga} - p}{|p_{\ga} - p|^2}\times \ga'(p_{\ga}).
$$

Recalling \eqref{derivativephi} and applying H\"older's inequality, we deduce that
\begin{equation}\label{difphiY}
\int_{T_{r_\ep}^\Omega(\ga)}\left| \nabla \varphi - Y_{\ga} \right |^2 =\int_{T_{r_\ep}^\Omega(\ga)}\left|h_{\ga}+\nabla f_\ga \right|^2
\\
\leq |T_{r_\ep}^\Omega(\ga)|^\frac13 \left\| h_{\ga}+\nabla f_{\ga}\right\|_{L^3(T_{r_\ep}^\Omega(\ga),\R^3)}^2\leq C r_\ep^\frac23=o_\ep(1),
\end{equation}
where we used the fact that $h_\ga+\nabla f_\ga \in L^3(T_{r_\ep}^\Omega(\ga),\R^3)$.

Using \eqref{difphiY}, and recalling that $|u_\ep| \leq 1$ and $\pmin \leq 1$, we deduce that
\begin{multline}\label{energysmalltube1}
	 \int_{T_{r_\ep}^\Omega(\ga)}\pmin^2\left(|\nabla u_\ep|^2+\frac1{2\ep^2}(1-|u_\ep|^2)^2\right)\\
	=\int_{T_{r_\ep}^\Omega(\ga)}\pmin^2\left(|\nabla |u_\ep||^2+|u_\ep|^2|\nabla \varphi|^2+\frac1{2\ep^2}(1-|u_\ep|^2)^2\right)
	\\
	=\int_{T_{r_\ep}^\Omega(\ga)}\pmin^2\left(|\nabla |u_\ep||^2+|u_\ep|^2|Y_\ga|^2+\frac1{2\ep^2}(1-|u_\ep|^2)^2\right)+o_\ep(1).
\end{multline}
Using the coordinates defined in \eqref{coordinates}, and recalling \eqref{volume}, we then find
\begin{multline*}
	\int_{T_{r_\ep}^\Omega(\ga)}\pmin^2\left(|\nabla |u_\ep||^2+|u_\ep|^2|Y_\ga|^2+\frac1{2\ep^2}(1-|u_\ep|^2)^2\right)\\
	=\int_0^{|\ga|} \left( \int_{D_s \cap \Omega }\pmin^2\left(|\nabla |u_\ep||^2+|u_\ep|^2|Y_\ga|^2+\frac1{2\ep^2}(1-|u_\ep|^2)^2\right)(1+O_\ep(r_\ep))\dd v\dd w\right)\dd s.
\end{multline*}
Combining this with \eqref{energysmalltube1}, we are led to
\begin{multline}\label{energysmalltube2}
	\int_{T_{r_\ep}^\Omega(\ga)}\pmin^2\left(|\nabla u_\ep|^2+\frac1{2\ep^2}(1-|u_\ep|^2)^2\right)
	\\
	\leq  (1+O_\ep(r_\ep))\int_0^{|\ga|} \left( \int_{D_s \cap \Omega }\pmin^2\left(|\nabla |u_\ep||^2+|u_\ep|^2|Y_\ga|^2+\frac1{2\ep^2}(1-|u_\ep|^2)^2\right)\dd v\dd w\right)\dd s +o_\ep(1).
\end{multline}
Let us now deal with the weight. From the hypothesis \eqref{hypothesispmin} on $\pmin$, we observe that for any point $x = x(v, w, s) \in D_s \cap \Omega$, we have, using once again that $\pmin\leq 1$,
\begin{multline*}
|\pmin^2(x)-\pmin^2(\ga(s))|=|\pmin(x)-\pmin(\ga(s))||\pmin(x)+\pmin(\ga(s))|\\
\leq 2r_\ep^\alpha\|\pmin\|_{C^{0,\alpha}(T_{r_\ep}^\Omega(\ga))}\leq 2C_1r_\ep^\alpha \lep^N.
\end{multline*}
Hence, using that $D_s\cap\Omega \subset D_s$, we find
\begin{multline}\label{energysmalltube3}
\int_0^{|\ga|} \left( \int_{D_s \cap \Omega }\pmin^2\left(|\nabla |u_\ep||^2+|u_\ep|^2|Y_\ga|^2+\frac1{2\ep^2}(1-|u_\ep|^2)^2\right)\dd v\dd w\right)\dd s\\
\leq \int_0^{|\ga|} \left(\pmin^2(\ga(s))+O_\ep(r_\ep^\alpha|\log\ep|^N)\right) \left( \int_{D_s }\left(|\nabla |u_\ep||^2+|u_\ep|^2|Y_\ga|^2+\frac1{2\ep^2}(1-|u_\ep|^2)^2\right)\dd v\dd w\right)\dd s.
\end{multline}
Notice that we conveniently defined $|u_\ep|$ slightly outside $\Omega$, so that we do not have to perform a special computation next depending on the endpoints of $\ga$, viewed as the original curve without extension beyond $\overline\Omega$.

To compute the integral over $D_s$ we use polar coordinates $(r,\theta)$ centered at $\ga(s)$. We have
$$
|\nabla |u_\ep|(x)|=\frac{\left|\degone'\left(\frac{r}{\ep}\right)\right|}{\ep \degone\left(\frac{r_\ep}\ep\right)}|\nabla r|=\frac{\left|\degone'\left(\frac{r}{\ep}\right)\right|}{\ep \degone\left(\frac{r_\ep}\ep\right)} \quad \mbox{and}\quad |Y_\ga|=\frac{r}{r^2}|\ga'(s)|=\frac1r.
$$
It follows that
\begin{multline*}
	\int_{D_s}\left(|\nabla |u_\ep||^2+|u_\ep|^2|Y_{\ga}|^2+\frac1{2\ep^2}(1-|u_\ep|^2)^2\right)\dd v\dd w\\
	=2\pi\int_0^{r_\ep}\( \frac{\degone'\left(\frac{r}{\ep}\right)^2}{\ep^2\degone\left(\frac{r_\ep}\ep\right)^2}+\frac{\degone\left(\frac{r}{\ep}\right)^2}{r^2\degone\left(\frac{r_\ep}\ep\right)^2}+\frac1{2\ep^2}\left(1-\frac{\degone\left(\frac{r}{\ep}\right)^2}{\degone\left(\frac{r_\ep}\ep\right)^2}\right)^2\)r\dd r\\
	=2\pi \int_0^{\frac{r_\ep}{\ep}}\( \frac{\degone'(t)^2}{\degone\left(\frac{r_\ep}\ep\right)^2}+\frac{\degone(t)^2}{t^2\degone\left(\frac{r_\ep}\ep\right)^2}+\frac1{2}\left(1-\frac{\degone(t)^2}{\degone\left(\frac{r_\ep}\ep\right)^2}\right)^2\)t\dd t,
\end{multline*}
where in the last equality we used the change of variables $r=\ep t$. Finally, using the fact that $\lim\limits_{\ep \to 0} \degone\left(\frac{r_\ep}{\ep}\right) = 1$, which follows from $\lim\limits_{\ep \to 0} \frac{r_\ep}{\ep} = +\infty$, we deduce from \eqref{asymptoticsf} that
$$
\frac12\int_{D_s}\left(|\nabla |u_\ep||^2+\rho_\ep^2|Y_\ga|^2+\frac1{2\ep^2}(1-|u_\ep|^2)^2\right)\dd v \dd w=\left(\pi \log \frac{r_\ep}{\ep}+\gamma +o_\ep(1)\right).
$$
By combining this with \eqref{energysmalltube2} and \eqref{energysmalltube3}, we are led to
\begin{multline*}
\frac12\int_{T_{r_\ep}^\Omega(\ga)}\pmin^2\left(|\nabla u_\ep|^2+\frac1{2\ep^2}(1-|u_\ep|^2)^2\right)\\
\leq \left(1+O_\ep(r_\ep)\right) \int_{0}^{|\ga|} \left(\pmin^2(\ga(s))+O_\ep(r_\ep^\alpha|\log\ep|^N)\right)\left(\pi \log \frac{r_\ep}{\ep}+\gamma +o_\ep(1)\right) \dd s\\
\leq \pi|\pmin^2\ga|\log \frac{r_\ep}{\ep}+|\pmin^2\ga|\gamma +o_\ep(1),
\end{multline*}
where in the last inequality we fixed $q=\frac{N+1}{\alpha}+\frac12>1$, which, recalling that $r_\ep=|\log\ep|^{-q}$, implies that
$$
O_\ep(r_\ep)\log \frac{r_\ep}{\ep}=o_\ep(1) \quad \mbox{and}\quad O(r_\ep^\alpha|\log\ep|^N)\log\frac{r_\ep}{\ep}=o_\ep(1).
$$
The claim \eqref{energyestimatesmalltube} is thus proved.

\item \textbf{Energy estimate outside $T_{r_\ep}^\Omega(\ga)$.} 

In the region $\Omega\setminus \overline{T_{r_\ep}(\ga)}$, we have $|u_\ep|\equiv 1$ and therefore, using  \eqref{eqjOmega} and \eqref{derivativephi}, we find
$$
|\nabla_A u_\ep|=|\nabla \varphi -A|=\left| X_\ga+\nabla f_\ga-A_\ga\right|=\left|j_\ga\right|.
$$
Using once again that $|u_\ep|\equiv 1$ in $\Omega\setminus \overline{T_{r_\ep}(\ga)}$ and that $\rho_\ep\leq 1$ in $\Omega$, we have
$$
\frac12 \int_{\Omega\setminus \overline{T_{r_\ep}(\ga)}}\rho_\ep^2|\nabla_A u_\ep|^2+\frac{\rho_\ep^4}{2\ep^2}(1-|u_\ep|^2)^2 +\frac12 \int_{\R^3}|\curl A|^2\leq I_\ep,
$$
where
$$
I_\ep\colonequals \frac12 \int_{\Omega\setminus \overline{T_{r_\ep}(\ga)}}|\nabla_A u_\ep|^2 +\frac12 \int_{\R^3}|\curl A|^2
=\frac12\int_{\Omega\setminus \overline{T_{r_\ep}(\ga)}}\left| j_\ga\right|^2+\frac12 \int_{\R^3}\left|\curl A_\ga\right|^2.
$$
Recalling \eqref{comega}, we obtain
\begin{equation*}
	I_\ep= C_{\Omega,\ga} +\pi |\ga| |\log r_\ep| +o_\ep(1).
\end{equation*}
Hence
\begin{equation}\label{energyestimateoutside}
\frac12 \int_{\Omega\setminus \overline{T_{r_\ep}(\ga)}}\rho_\ep^2|\nabla_A u_\ep|^2+\frac{\rho_\ep^4}{2\ep^2}(1-|u_\ep|^2)^2 +\frac12 \int_{\R^3}|\curl A|^2\leq \pi|\ga| |\log r_\ep| +C_{\Omega,\ga}
+o_\ep(1).
\end{equation}

\item \textbf{Final energy estimate.}

By combining the main estimates from the previous steps, namely \eqref{covariantderivativeenergy}, \eqref{energyestimatesmalltube}, and \eqref{energyestimateoutside}, we are led to
\begin{align*}
\fen(u_\ep,A)=\frac12 \int_\Omega \pmin^2|\nabla_A u_\ep|^2 +&\frac{\pmin^4}{\ep^2}(1-|u_\ep|^2)^2+\frac12 \int_{\R^3}|\curl A|^2\\
&\leq \pi|\pmin^2\ga|\lep +\pi|\ga||\log r_\ep|+C_{\Omega,\ga}+|\pmin^2\ga|\gamma+o_\ep(1)\\
&\leq \pi|\pmin^2\ga|\lep +\left(\frac{N+1}{\alpha}+1\right)\pi|\ga|\log \lep+C_{\Omega,\ga}+o_\ep(1),
\end{align*}
where in the last inequality we used that $r_\ep = \lep^{-\left(\frac{N+1}{\alpha}+\frac12\right)}$, and therefore 
$$
\pi|\ga||\log r_\ep|+|\pmin^2\ga|\gamma\leq \pi|\ga||\log r_\ep|+|\ga|\gamma\leq \left(\frac{N+1}{\alpha}+1\right)\pi|\ga|\log \lep.
$$
This completes the proof of \eqref{construction:upperbound}.

\item\label{Step5} \textbf{Vorticity estimate.}
Let $B\in C^{0,1}_T(\Omega,\R^3)$. By integration by parts, recalling \eqref{defvorticity} and that $|u_\ep|\equiv 1$ in $\Omega\setminus T_{r_\ep}(\ga)$, we have (recall \eqref{TOmega})
\begin{align}\label{vortest1}
	\int_\Omega\mu(u_\ep,A)\cdot B&= \int_\Omega \curl \(j(u_\ep,A)+A\)\cdot B =\int_\Omega \left( j(u_\ep,A)+A\right) \cdot \curl B\notag \\
	&=\int_\Omega |u_\ep|^2\nabla \varphi\cdot \curl B+\int_\Omega (1-|u_\ep|^2)A \cdot \curl B\notag \\
	&=\int_\Omega \nabla \varphi \cdot \curl B + \int_{T_{r_\ep}^\Omega(\ga)} (1-|u_\ep|^2)(A -\nabla \varphi)\cdot \curl B.
\end{align}
We start by estimating the second integral in the RHS of \eqref{vortest1}. Since $A -\nabla \varphi=-j_\ga \in L^q(\Omega,\R^3)$ for $q<2$, using that $0\leq 1-|u_\ep|^2\leq 1$ and $b\leq \pmin^2\leq 1$, by applying H\"older's inequality we deduce that
\begin{align}
	\left|\int_{T_{r_\ep}^\Omega(\ga)} (1-|u_\ep|^2) j_\ga \cdot \curl B\right|
	&\leq \|\curl B\|_{L^\infty(\Omega,\R^3)}\|j_\ga\|_{L^{\frac32}(\Omega,\R^3)}\|(1-|u_\ep|^2)\|_{L^3(T_{r_\ep}^\Omega(\ga),\R^3)} \notag \\
	&\leq C\|\curl B\|_{L^\infty(\Omega,\R^3)}\|(1-|u_\ep|^2)\|_{L^2(T_{r_\ep}^\Omega(\ga),\R^3)}^\frac{2}{3} \notag \\
	&\leq C\|\curl B\|_{L^\infty(\Omega,\R^3)}\ep^\frac{2}{3}\left(\int_\Omega \frac{\pmin^4}{\ep^2}(1-|u_\ep|^2)^2\right)^\frac13 \notag \\
	&\leq C\|\curl B\|_{L^\infty(\Omega,\R^3)}\ep^\frac{2}{3}\lep^\frac13,\label{vortest2}
\end{align}
where in the last inequality we used that
$$
\int_\Omega \frac{\pmin^4}{\ep^2}(1-|u_\ep|^2)^2\leq C|\log \ep|
$$ 
in view of \eqref{energyestimatesmalltube}.

\medskip
We now proceed to compute the first integral in the RHS of \eqref{vortest1}. Using \eqref{derivativephi} and an integration by parts, we obtain
\begin{equation}\label{vortest3}
	\int_\Omega \nabla \varphi \cdot \curl B=  \int_\Omega\left(X_\ga+\nabla f_\ga\right)\cdot \curl B= \int_\Omega \curl X_\ga \cdot B.
\end{equation}
Finally, recalling \eqref{eqXga}, from \eqref{vortest1}, \eqref{vortest2}, and \eqref{vortest3}, we deduce that
\begin{equation*}
	\left\|\mu(u_\ep,A)-2\pi \ga \right\|_{(C_T^{0,1}(\Omega,\R^3))^*}\leq C\ep^\frac{2}{3}\lep^\frac13.
\end{equation*}
This concludes the proof of \eqref{construction:vorticityestimate} for $\beta=1$. The proof of the estimate for $\beta\in(0,1)$ follows from interpolation. In fact, from \cite{roman-vortex-construction}*{Lemma 8.1}, we have that
$$
\|\mu(u_\ep,A)\|_{(C_0^0(\Omega,\R^3))^*}\leq C\left(\frac12 \int_{\Omega}|\nabla_A u_\ep|^2+\frac1{2\ep^2}(1-|u_\ep|^2)^2+|\curl A|^2\right).
$$
Since $\rho_\ep \geq b$, we deduce that the RHS is bounded above by $C\fen(u_\ep,A)$. Using \eqref{construction:upperbound} and $\|\ga\|_{(C^0(\Omega))^*}\leq C$, we then deduce that
$$
\|\mu(u_\ep,A)-2\pi\ga\|_{(C_0^0(\Omega,\R^3))^*}\leq C\lep.
$$ 
To conclude, we use interpolation\footnote{This relies on the classical interpolation result of Jerrard--Soner \cite{JerSon}*{Lemma 3.3}, formulated for the dual of $C_0^{0,\beta}(\Omega,\R^3)$. The corresponding estimate in the dual space associated with $C_T^{0,\beta}(\Omega,\R^3)$, which is suited to our setting, was first established in \cite{JerMonSte}*{Proposition 3.1}.}
\begin{multline*}
\left\|\mu(u_\ep,A)-2\pi \ga \right\|_{(C_T^{0,\beta}(\Omega,\R^3))^*}\\ \leq \left\|\mu(u_\ep,A)-2\pi \ga \right\|_{(C_T^{0,1}(\Omega,\R^3))^*}^\beta \left\|\mu(u_\ep,A)-2\pi \ga \right\|_{(C_0^0(\Omega,\R^3))^*}^{1-\beta}\leq C\ep^{\frac{2}{3}\beta}|\log\ep|^{1-\frac23 \beta}.
\end{multline*}
This concludes the proof of \eqref{construction:vorticityestimate} for $\beta\in(0,1)$.
\end{enumerate}
\end{proof}

\begin{remark}\label{rmk:uniformity}
	Let $(\ga_n)$ be a family of curves in $\Omega$ such that $|\Gamma_n|\leq C$ uniformly in $n$. Then the $o_\varepsilon(1)$ term appearing in \eqref{construction:upperbound} for $\Gamma=\Gamma_n$ is uniform in $n$. Similarly, the constant $C$ in \eqref{construction:vorticityestimate} can be chosen independently of $n$.
	
	Indeed, the $o_\varepsilon(1)$ error terms in the proof of \eqref{construction:upperbound} are either proportional to the length of the curves, or involve quantities that remain uniformly bounded for such a family. Since all curves $\ga_n$ lie in $\Omega$ and have uniformly bounded length, the $L^\frac32(\Omega)$ and $L^\frac32(\partial\Omega)$ norms of $X_{\ga_n}$ are uniformly bounded in $n$. By \cite{RSS2}*{(2.6)}, this implies that the corresponding $L^6(\Omega)$ norms of $A_{\ga_n}$ and $L^3(\Omega)$ norms of $\nabla\phi_{\ga_n}$ are also bounded uniformly in $n$, which are the quantities that enter the $o_\varepsilon(1)$ terms in Step~\ref{Step1} of the proof.
	
	Arguing similarly as above, one proves that the $L^3(\Omega)$ norm of $h_{\ga_n}+\nabla f_{\ga_n}$, which enters the $o_\varepsilon(1)$ term in Step~\ref{Step2}, is uniformly bounded in $n$. The remaining $o_\varepsilon(1)$ terms in the proof of \eqref{construction:upperbound} are not hard to see to be proportional to the length of $\ga_n$, and therefore are uniformly bounded in $n$ as well.
	
	Finally, the same type of argument applies in Step~\ref{Step5}, ensuring that the constant in \eqref{construction:vorticityestimate} can be chosen uniformly in $n$, since both the $L^\frac32(\Omega)$ norm of $j_{\ga_n}$ and the dual norm $\|\Gamma_n\|_{C^0(\Omega)^*}$ are uniformly bounded in $n$.
\end{remark}

\section{Upper bound for the first critical field}\label{sec:proofmain2}
In this section we provide a proof for Theorem~\ref{thm:firstcriticalfield}.
\begin{proof}[Proof of Theorem~\ref{thm:firstcriticalfield}]
We consider, for any $\ep$ sufficiently small, curves $\ga_\ep$ such that \eqref{almostopt}, \eqref{boundlength}, and \eqref{boundCOmega} hold. Notice in particular that, from \eqref{boundlength}, we have 
$$
|\pmin^2 \ga_\ep| \leq |\ga_\ep| \leq C_0.
$$
By applying Theorem~\ref{thm:upperbound} with $\ga=\ga_\ep$, we obtain a pair $(u_\ep,A_\ep)\in H^1(\Omega,\C)\times H^1(\R^3,\R^3)$ such that\footnote{By Remark~\ref{rmk:uniformity}, the $o_\ep(1)$ term is uniform with respect to the curves $\ga=\ga_\ep$.}
\begin{multline}\label{enspl1}
\fen(u_\ep,A_\ep)=\frac12\int_\Omega\pmin^2|\nabla_{A_\ep} u_\ep|^2+\frac{\pmin^4}{2\ep^2}(1-|u_\ep|^2)^2+\frac12\int_{\R^3}|\curl A_\ep|^2
\\ \leq \pi |\pmin^2 \ga_\ep|\lep +\left(\frac{N+1}{\alpha}+1\right)\pi |\ga_\ep|\log\lep+C_{\Omega,\ga_\ep}+o_\ep(1)
\end{multline}
and, for any $\beta\in(0,1]$,\footnote{By Remark~\ref{rmk:uniformity}, the constant $C$ can be chosen uniformly with respect to the curves $\ga=\ga_\varepsilon$.}
\begin{equation}\label{vortestconfig}
\|\mu(u_\ep,A_\ep)-2\pi \ga_\ep\|_{(C_T^{0,\beta}(\Omega,\R^3))^*}\leq C\ep^{\frac23 \beta}\lep^{1-\frac23 \beta} .
\end{equation}

\smallskip
We now consider the configuration
$$
(\u_\ep,\A_\ep)=(\meisdecom),
$$
where $(\meisconf)$ is the Meissner state. Using the energy splitting \eqref{energy splitting equation}, we find
\begin{equation}\label{enspl2}
GL_\ep(\u_\ep,\A_\ep) = GL_\ep(\pmin e^{i h_\ex \phi_\ep}, h_\ex A^0_\ep) + \fen(u_\ep,A_\ep) - h_\ex \int_\Omega \mu(u_\ep,A_\ep) \cdot B^0_\ep+\err. 
\end{equation}
Let us estimate $\err$. From \eqref{err}, using H\"{o}lder's inequality and $b\leq \pmin^2\leq 1$, we are led to
$$
\err\leq C h_\ex^2 \ep \|\curl B_\ep^0\|_{L^4(\Omega)}^2\left(\int_\Omega \frac{\pmin^4}{\ep^2}(1-|u_\ep|^2)^2\right)^\frac12.
$$
From \eqref{construction:upperbound}, we find
$$
\int_\Omega \frac{\pmin^4}{\ep^2}(1-|u_\ep|^2)^2\leq C\lep.
$$
On the other hand, from \cite{diaz-roman-3d}*{Proposition 2.2}, we have
$$
\|\curl B_\ep^0\|_{L^4(\Omega)}^2\leq C,
$$
where $C$ does not depend on $\ep$. 
Since we assume \eqref{hypothesish_ex}, we then deduce that
$$
\err\leq C \ep^{1-2\eta}\lep^\frac12=o_\ep(1).
$$

Using \eqref{vortestconfig}, \eqref{hypothesish_ex}, and the fact that $B_\ep^0\in C_T^{0,1}(\Omega,\R^3)$ with $\|B_\ep^0\|_{C_T^{0,\beta}}\leq C$ for any $\beta\in (0,1)$, where the constant $C$ does not depend on $\ep$ (see once again \cite{diaz-roman-3d}*{Proposition 2.2}), we find
$$
 h_\ex \int_\Omega \mu(u_\ep,A_\ep) \cdot B_\ep^0=2\pi h_\ex \ip{\ga_\ep}{B_\ep^0}+o_\ep(1)=2\pi h_\ex\RR_{\pmin^2}(\ga_\ep)|\pmin^2 \ga_\ep|+o_\ep(1).
$$
Inserting $h_\ex\geq H_{c_1}^\ep+K^0\log\lep=\frac{\lep}{2\RR_\ep}+K^0\log\lep$ and using \eqref{almostopt}, we then obtain
$$
h_\ex \int_\Omega \mu(u_\ep,A_\ep) \cdot B_\ep^0\geq \pi|\pmin^2 \ga_\ep|\lep+2\pi K^0 \RR_\ep |\pmin^2 \ga_\ep|\log\lep +O_\ep(1).
$$
Using this together with \eqref{enspl1} in \eqref{enspl2}, we deduce that
\begin{multline}\label{intermediate}
GL_\ep(\u_\ep,\A_\ep)\leq GL_\ep(\pmin e^{i h_\ex \phi_\ep}, h_\ex A^0_\ep) \\
+\(\frac{N+1}{\alpha}+1\)\pi|\ga_\ep| \log\lep +C_{\Omega,\ga_\ep}-2\pi K^0\RR_\ep |\pmin^2 \ga_\ep|\log\lep +O_\ep(1).
\end{multline}

\smallskip
Let us now show that there exists $c_0>0$ (independent of $\ep$) such that
\begin{equation}\label{boundbelowlength}
|\pmin^2\ga_\ep|\geq c_0>0.
\end{equation}
We start by closing $\ga_\ep$ by connecting its endpoints with a curve lying on $\partial \Omega$. More precisely, we consider an arbitrary smooth curve on $\partial \Omega$ that connects the endpoints of $\ga_\ep$, oriented consistently with the orientation of $\ga_\ep$ and such that, if we denote by $\tilde\ga_\ep$ the resulting loop, we have $|\tilde \ga_\ep|\leq C|\ga_\ep|$. Since $B_\ep^0 \times \nu = 0$ on $\partial \Omega$, by Stokes' theorem, we have
$$
\ip{\ga_\ep}{B_\ep^0}=\ip{\tilde\ga_\ep}{B_\ep^0}=\int_{S_{\ga_\ep}} \curl B_\ep^0,
$$
where $S_{\ga_\ep}$ denotes a surface with least area among those whose boundary is $\tilde \ga_\ep$, i.e. a solution to the associated Plateau's problem. 

By H\"{o}lder's inequality and the isoperimetric inequality, we have
$$
\int_{S_{\ga_\ep}}|\curl B_\ep^0|\leq \|\curl B_\ep^0\|_{L^4(\Omega,\R^3)} \mathrm{Area}(S_{\ga_\ep})^\frac34\leq C (|\tilde \ga_\ep|^2)^\frac34\leq C|\ga_\ep|^\frac32 \leq C |\pmin^2 \ga_\ep|^\frac32.
$$ 
Therefore,
$$
\RR_{\pmin^2}(\ga_\ep) \leq C|\pmin^2\ga_\ep|^\frac12,
$$
which combined with $\liminf\limits_{\ep\to0}\RR_\ep>0$ and \eqref{almostopt} yields the claim.

\smallskip
By combining \eqref{boundbelowlength}, \eqref{boundlength}, and \eqref{boundCOmega} with \eqref{intermediate}, we are led to
\begin{multline*}
GL_\ep(\u_\ep,\A_\ep)\leq GL_\ep(\pmin e^{i h_\ex \phi_\ep}, h_\ex A^0_\ep) \\
+\left(\(\frac{N+1}{\alpha}+1\)\pi C_0 +C_0-2\pi c_0 K^0\frac12 \liminf_{\ep\to 0}\RR_\ep \right) \log\lep +O_\ep(1),
\end{multline*}
where we also used that $R_\ep\geq \frac12 \liminf_{\ep\to 0}R_\ep>0$ for any $\ep$ sufficiently small.

Hence, provided
$$
K^0>\frac{\(\frac{N+1}{\alpha}+1\)\pi C_0 +C_0}{\pi c_0\liminf\limits_{\ep\to 0}\RR_\ep}+1, 
$$ 
we have 
\begin{equation}\label{finalest}
GL_\ep(\u_\ep,\A_\ep)\leq GL_\ep(\pmin e^{i h_\ex \phi_\ep}, h_\ex A^0_\ep) -\log\lep +O_\ep(1).
\end{equation}
On the other hand, by \cite{diaz-roman-3d}*{Theorem 1.2}, since we assume $h_\ex\leq \ep^{-\eta}$, for $\eta\in \(0,\frac12\)$ (recall \eqref{hypothesish_ex}), we know that if $(\u,\A)$ is any configuration without vortex lines --- that is, $|\u| > c > 0$ for some $c > 0$ --- such that
$$
GL_\ep(\u,\A) \leq GL_\ep(\pmin e^{i h_\ex \phi_\ep}, h_\ex A^0_\ep),
$$
then necessarily
$$
GL_\ep(\u,\A)=GL_\ep(\pmin e^{i h_\ex \phi_\ep}, h_\ex A^0_\ep)+o_\ep(1),
$$
where the $o_\varepsilon(1)$ term is uniform over vortexless configurations (in particular, independently of $c>0$).
Therefore, the minimal energy among vortexless configurations coincides with 
$$
GL_\ep(\pmin e^{i h_\ex \phi_\ep}, h_\ex A^0_\ep)
$$
up to $o_\ep(1)$. Combining this with \eqref{finalest}, yields
$$
\inf_{(\u,\A)\ \text{vortexless}} GL_\ep(\u,\A)=GL_\ep(\pmin e^{i h_\ex \phi_\ep}, h_\ex A^0_\ep)+o_\ep(1)>GL_\ep(\u_\ep,\A_\ep)+\frac12 \log\lep.
$$
We have thus constructed a configuration $(\u_\ep,\A_\ep)$ with a single vortex line located at $\ga_\ep$, such that its energy is strictly less than any configuration without vortex lines. Hence the global minimizer of the energy $GL_\ep$ must have vortex lines provided
$$
h_\ex \geq H_{c_1}^\ep +K^0\log\lep.
$$
This concludes the proof.
\end{proof}

\section*{Compliance with Ethical Standards}
The author declares that there are no conflicts of interest.
\bibliography{references3DPinning}
\end{document}